\documentclass[12pt]{amsart}
\usepackage{amsfonts}
\usepackage{graphicx,amssymb,amsfonts,amsmath,amscd}
\usepackage[all,ps,cmtip]{xy}
\usepackage{amscd}

\usepackage{mathrsfs}

\usepackage{eucal}
\usepackage{hyperref}

\newtheorem{defn}{Definition}[section]
\newtheorem{thm}[defn]{Theorem}
\newtheorem{lem}[defn]{Lemma}

\newtheorem*{prob}{Problem}
\newtheorem{cor}[defn]{Corollary}
\newtheorem{prop}[defn]{Proposition}
\newtheorem{rem}[defn]{Remark}

\newcommand{\N}{\mathbb{N}}
\newcommand{\R}{\mathbb{R}}
\newcommand{\C}{\mathbb{C}}
\newcommand{\Z}{\mathbb{Z}}

\newcommand{\T}{\mathbb{T}}
\newcommand{\D}{\mathbb{D}}
\newcommand{\Prob}{\mathbb{P}}

\newcommand{\E}{\mathbb{E}}

\input xy
\xyoption{all}

\begin{document}
\title{On the \text{UMD} constants for a class of iterated $L_p(L_q)$ spaces}
\author{Yanqi QIU}
\thanks{The author was partially supported by ANR grant 2011 BS01 008 01}
\address{Inst. Math. Jussieu, \'Equipe d'Analyse Fonctionnelle}
\address{Universit\'e Paris VI, 75252 Paris Cedex 05, France}
\email{yanqi-qiu@math.jussieu.fr}
\begin{abstract}
Let $1 < p \neq q < \infty$ and $(D, \mu) = (\{\pm 1\}, \frac{1}{2} \delta_{-1} + \frac{1}{2}\delta_1)$. Define by recursion: $X_0 = \C$ and $X_{n+1} = L_p(\mu; L_q(\mu; X_n))$. In this paper, we show that there exist $c_1=c_1(p, q)>1$ depending only on $p, q$ and $ c_2 = c_2(p, q, s)$ depending on $p, q, s$, such that the $\text{UMD}_s$ constants of $X_n$'s satisfy $c_1^n \leq C_s(X_n) \leq c_2^n$ for all $1 < s < \infty$. Similar results will be showed for the analytic $\text{UMD}$ constants. We mention that the first super-reflexive non-$\text{UMD}$ Banach lattices were constructed by Bourgain. Our results yield another elementary construction of super-reflexive non-$\text{UMD}$ Banach lattices, i.e. the inductive limit of $X_n$, which can be viewed as iterating infinitely many times $L_p(L_q)$.
\end{abstract}
\maketitle

%%%%%%%%%%%%%%%%%%%%%%%%%%Introduction

\section{Introduction}
A Banach space $X$ is $\text{UMD}$ if for all (or equivalently, for some) $1 < s < \infty$ there is a constant $C > 0$ depending only on $s$ and $X$ such that
\begin{eqnarray}\label{definition_umd}
\sup_{\varepsilon_k \in \{-1, 1\}} \| \sum_{k=0}^n \varepsilon_k df_k \|_{L_s(X)} \leq C \| \sum_{k=0}^n df_k \|_{L_s(X)}
\end{eqnarray}
for all $n \geq 0$ and all $X$-valued martingale difference sequences $(df_k)_{k=0}^n$. The best such $C$ is called the $\text{UMD}_s$ constant of $X$ and will be denoted by $C_s(X)$ in the sequel. It is well-known that in the above definition, we can restrict to the dyadic martingale differences and the best constant remains the same. The $\text{UMD}$ property for Banach spaces was introduced by Maurey and Pisier. The reader is refered to Burkholder's papers \cite{Burkholder_geo_UMD, Burkholder3} for the details of the $\text{UMD}$ property.

Let $\T = \{ z \in \C: | z | =1 \}$ be the one dimensional torus equipped with the normalised Haar measure $m$. Consider the canonical filtration on the probability space $(\T^\N, m^{\otimes \N})$ defined by $$ \sigma(z_0) \subset \sigma(z_0, z_1) \subset \cdots \subset \sigma(z_0, z_1, \cdots, z_n) \subset \cdots.$$ By definition, a Hardy martingale in $L_s(\T^\N; X)$ is a martingale $f = (f_n)_{n\geq 0}$ with respect to the canonical filtration such that $\sup_n \| f_n \|_{L_s} < \infty$, and such that the martingale difference $df_n = f_n -f_{n-1}$ (by convention, $d f_0: = f_0$) is analytic in the last variable $z_n$, i.e., $df_n$ has  the form:$$df_n (z_0, \cdots, z_{n-1}, z_n) = \sum_{k\ge 1} \phi_{n, k} (z_0, \cdots, z_{n-1}) z_n^{k}.$$
In the above definition of $\text{UMD}$ spaces, if the Banach space is over the complex field $\C$, and if we restrict to the Hardy martingales, then a different class of Banach spaces is defined, i.e. the analytic $\text{UMD}$ class ($\text{AUMD}$ by abreviation). The best constant is called the $\text{AUMD}_s$ constant of $X$ and will be denoted by $C_s^{a}(X)$. Note that $\text{UMD}$ implies $\text{AUMD}$ but not conversely, for instance, $L_1(\T, m)$ is an $\text{AUMD}$ space which is not $\text{UMD}$ (cf. \cite{Garling_UMD}).

It is well-known that $\text{UMD}$ implies super-reflexivity but not conversely. The first super-reflexive non-$\text{UMD}$ Banach space was constructed by Pisier in \cite{Pisier_superR_nonUMD}. Super-reflexive non-$\text{UMD}$ Banach lattices were later constructed by Bourgain in \cite{Bourgain, Bourgain_UMD}. We refer to Rubio de Francia's paper \cite{Rubio_de_Francia_mart} for some open problems related to the super-reflexive non-$\text{UMD}$ Banach lattices.

The main topic of this paper is the investigation of the $\text{UMD}$ constants of a family of iterated $L_p(L_q)$-spaces. As a consequence of our results, we give an elementary construction of super-reflexive non-$\text{UMD}$ Banach lattices.

%%%%%%%%%%%%%%%%%%%%%%%%%%%%%%%%%%%%ELEMENTARY Inequalities

\section{ Some elementary inequalities} 

We will use the following lemma.
\begin{lem}\label{lem}
Let $(\Omega, \nu)$ be a measure space such that $\nu$ is finite. Suppose that $\alpha \ne 1$ and $ 0 < \alpha < \infty$. If $F, f \in L_{\alpha}(\Omega, \nu) \bigcap L_1(\Omega, \nu)$ satisfy $$\int (| F |+|g|)^{\alpha} d \nu \leq \int (| f | + | g |)^{\alpha} d \nu$$ for all $ g \in L_\infty(\Omega, \nu)$. Then $| F | \leq | f |$ a.e..
\end{lem}
\begin{proof}
Consider first those $ g \in L_\infty(\Omega, \nu)$ such that there exists $\delta > 0$ and $ | g | \geq \delta$ a.e.. If $F, f$ satisfy the condition in the statement, then for all $\varepsilon > 0$, we have \begin{eqnarray}\label{variation}\int (\varepsilon | F |+|g|)^{\alpha} d \nu \leq \int (\varepsilon | f | + | g |)^{\alpha} d \nu.\end{eqnarray} 
By the mean value theorem, there exists $ \theta = \theta_\varepsilon \in (0, 1)$, such that $$\frac{(\varepsilon | f | + | g | )^{\alpha} - | g |^{\alpha}}{\varepsilon} = \alpha (\theta \varepsilon | f | + | g |)^{\alpha -1} | f |.$$ If $\alpha < 1$, then $ (\theta \varepsilon | f | + | g |)^{\alpha -1} | f | \leq | g|^{\alpha-1} | f | \in L_1(\Omega, \nu)$ and if $\alpha > 1$, then for $0 < \varepsilon <1$, we have $0< \theta \varepsilon < 1$ and hence $(\theta \varepsilon | f | + | g |)^{\alpha -1} | f | \leq 2^{\alpha-1}(| f |^{\alpha} + | g |^{\alpha-1} | f |) \in L_1(\Omega, \nu)$. By the dominated convergence theorem, we have $$ \lim_{\varepsilon \to 0^{+}} \frac{\int (\varepsilon | f | + | g |)^{\alpha} d \nu - \int | g |^{\alpha} d \nu }{\varepsilon} = \alpha \int | f | | g |^{\alpha -1} d \nu.$$ The same equality holds for $F$. Combining this with \eqref{variation}, we get $$\int | F | | g |^{\alpha-1} d \nu \leq \int | f | |g |^{\alpha-1} d \nu.$$ Replacing $g$ by $|g|^{\frac{1}{\alpha-1}}$ yields $$\int | F | | g | d \nu \leq \int | f | |g | d \nu.$$ By approximation, the above inequality holds for all $g \in L_\infty(\Omega, \nu)$. Hence $|F| \leq |f |$ a.e., as announced.
\end{proof}

\begin{prop}\label{genlem}
Let $(\Omega, \nu)$ be a measure space such that $\nu$ is finite. Suppose that $ 1 \leq p \ne q < \infty$. If $F, f \in L_p(\Omega, \nu) \bigcap L_q(\Omega, \nu)$ satisfy $$\int (| F | ^q + |g|^q)^{p/q} d \nu \leq \int (| f |^q + | g| ^q)^{p/q} d \nu$$ for all $g \in L_\infty(\Omega, \nu)$. Then $|F| \leq | f |$ a.e..
\end{prop}
\begin{proof}
This is just a reformulation of Lemma \ref{lem}.
\end{proof}

Let $D = \{-1, 1\}$ be the Bernoulli probability space equipped with the measure $\mu = \frac{1}{2} \delta_{-1} + \frac{1}{2}\delta_1$. For any $1\leq q \leq \infty$, the 2-dimensional $\ell_q$-space will be denoted by $\ell_q^{(2)}$.

\begin{prop}\label{mainlem}
Suppose that $1 \leq p\ne q \leq \infty$. Let $P$ be the projection on $L_p(\mu;\ell_q^{(2)})$ defined by $$\begin{array}{cccc} P: &L_p(\mu; \ell_q^{(2)}) & \rightarrow & L_p(\mu; \ell_q^{(2)}) \\ & (f, g) & \mapsto & (\E f, g)\end{array},$$ where $\E$ is the expectation.
Then $P$ is not contractive.
\end{prop}
\begin{proof}
Assume first that both $p, q$ are finite. If $P$ is contractive, then for any two functions $f $ and $g$, we have $$ \int (| \E f |^q + |g|^q )^{p/q} d\mu \leq \int  (| f |^q + |g|^q )^{p/q} d\mu.$$
By Proposition \ref{genlem}, it follows that $|\E (f) |\leq |f|$, which is a contradiction, hence $P$ is not contractive. 

If $ p = \infty$ and $ 1< q < \infty$, then $p' =1$ and $ 1< q' <\infty$. Since the adjoint map $P^*$ on $L_1(\mu; \ell_{q'}^{(2)})$ has the same form as $P$, the preceding argument shows that $P^*$ and hence $P$ is not contractive.

If $p = \infty$ and $ q = 1$. Assume $P$ is contractive, then we have \begin{eqnarray}\label{extreme}\big \| | \E f | + | g | \big\|_\infty \leq \big \| | f | + | g | \big \|_\infty.\end{eqnarray} Consider $ f = 1 + \varepsilon, g = 1-\varepsilon$, where $\varepsilon: D \rightarrow D$ is the identity function. Then the left hand side of \eqref{extreme} equals to 3 while the right hand side equals to 2. This contradiction shows that $P$ is not contractive.

If $1\leq p < \infty$ and $q = \infty$, then $ 1< p' \leq \infty$ and $q' = 1$, hence $P^*$ is not contractive. It follows that $P$ is not contractive. 
\end{proof}

The norm of $P$ on $L_p(\mu; \ell_q^{(2)})$ will be denoted by $c(p, q)$ in the sequel. If $p=q$, then $c(p,p) = 1$. If $1 \leq p \neq q \leq \infty$, then \begin{eqnarray} \label{constant} c(p,q) > 1. \end{eqnarray}

\begin{rem}
It is not difficult to check that $c(\infty, 1) = c(1, \infty) = \frac{3}{2}$. But we do not know the exact value of $c(p,q)$ for general $p \ne q$.
\end{rem}

As usual, we set $$ H_p(\T) = \{ f \in L_p(\T, m): \hat{f}(k) = 0, \forall k \in \Z_{ < 0} \}.$$ We will say that a measurable function $f : \T \rightarrow \C$ is bounded from below, if there exists $\delta > 0$, such that $| f | \geq \delta$ a.e. on $\T$. If $f\in L_p(\T)$ is bounded from below, then the geometric mean $M(|f|)$ of $|f|$ is defined by $$\log M(|f|) = \int_\T \log | f(z)| dm(z).$$ In particular, if $f: \D \rightarrow \C$ is an outer function, then \begin{eqnarray}\label{geo_mean_outer} M(|f|) = |f(0)| = |\E f|. \end{eqnarray}

The following elementary proposition will be used in \S \ref{analytic_umd} when we treat the analytic $\text{UMD}$ property.
\begin{prop}\label{geo_mean}
Suppose that $1 \leq p\neq q < \infty$. Define $\kappa(p, q)$ to be the best constant $C$ satisfying the property:  For any measurable partition $\T = A \dot{\cup} B$ with $m(A) = m(B) = \frac{1}{2}$, for any function $ f= f_1 \chi_A + f_2 \chi_B$ with $f_1>0, f_2 > 0$ and any function $g = g_1 \chi_A + g_2\chi_B$, we have $$ \int_\T  (M(|f|)^q + | g |^q)^{p/q} dm \leq C^p  \int_\T (|f|^q + | g |^q)^{p/q} dm.$$ Then $\kappa(p,q) > 1$.
\end{prop}
\begin{proof}
Assume $k(p, q) \leq 1$. Fix any measurable partition $\T = A \dot{\cup} B$ such that $m(A) = m(B) = \frac{1}{2}$. Consider the 2-valued functions $f = f_1\chi_A + f_2 \chi_B$ and $g = g_1 \chi_A + g_2 \chi_B$ with $f_1, f_2$ positive scalars. By Proposition \ref{genlem}, $M(f) \leq  f $. However, one can easily check that $M( f ) = f_1^{1/2} f_2^{1/2}$. If $f_1 > f_2$, then $M( f ) >   f_2^{1/2}  f_2^{1/2} =  f_2 $, which contradicts to $ M(f) \leq f$. Whence the announced statement.
\end{proof}

%%%%%%%%%%%%%%%%%%%%%%%%%%%%%%%%%%%%%%%%UMD case

\section{\text{UMD} constants of iterated $L_p(L_q)$ spaces}\label{mainresults}
The following definition is essential in the sequel.
\begin{defn}
Consider a Banach space $X$ with a fixed family of vectors $\{ x_i\}_{i \in I}$. We define $S(X; \{ x_i \})$ to be the best constant $C$ such that \begin{eqnarray}\label{definition_S(X)} \Big \| \sum_{k=0}^N \E^{\mathcal{A}_k} (\theta_k) x_{i_k} \Big \|_{L_1(\Omega, \Prob; X)} \leq C \Big \| \sum_{k=0}^N \theta_k x_{i_k} \Big \|_{L_\infty(\Omega, \Prob; X)}\end{eqnarray} holds for any $N\in \N$, any probability space $(\Omega, \mathcal{F}, \Prob)$ equipped with a filtration $\mathcal{A}_0 \subset \mathcal{A}_1 \subset \cdots \subset \mathcal{A}_n \subset \cdots \subset \mathcal{F}$, any $N+1$ distinct indices $\{ i_0, i_1, \cdots, i_N\} \subset I$ and any $N+1$ functions $\theta_0, \theta_1, \cdots, \theta_N$ in $L_\infty(\Omega, \mathcal{F}, \Prob)$. 

If there does not exist such constant,  we set $S(X; \{x_i\}) = \infty$. 
\end{defn}
In what follows, we are mostly interested in the special case when $\{ x_i \}$ is a 1-unconditional basic sequence, since in this case we can relate $S(X; \{ x_i \})$ to the $\text{UMD}$ constants of $X$. If $\{ x_i \}$ is clear from the context and there is no confusion, we will use the simplified notation $S(X)$ for $S(X; \{ x_i\})$. In particular, if $X$ has a natural basis, then $S(X)$ will always mean to be calculated with this basis.

We will need the following well-known Stein inequality in $\text{UMD}$ spaces, which was originally proved by Bourgain \cite{Bourgain_H1_BMO}. For the sake of completeness, we include the proof.
\begin{thm}\label{stein}
Let $X$ be a $\text{UMD}$ space. Then for any $1 < s <\infty$, any finite sequences of functions $(F_k)_{k \geq 0}$ in $L_s(\Omega, \Prob; X)$ and any filtration $\mathcal{A}_0 \subset \mathcal{A}_1 \subset \cdots \subset \mathcal{A}_n \subset \cdots$ on $(\Omega, \Prob)$, we have \begin{eqnarray}\label{stein_umd} \Big \| \sum_k \varepsilon_k \E_k (F_k) \Big\|_{L_s(\mu_\infty \times  \Prob; X)} \leq C_s(X) \Big \| \sum_k \varepsilon_k F_k \Big \|_{L_s(\mu_\infty \times  \Prob; X)},\end{eqnarray} where $\E_k = \E^{\mathcal{A}_k}$ and $(\varepsilon_k)_{k \geq 0}$ is the usual Rademacher sequence on $(D^{\N}, \mu_\infty)$, $\mu_\infty = \mu^{\otimes \N}$.
\end{thm}
\begin{proof}
Let $f = \sum_k \varepsilon_k F_k$ and  $f' = \sum_k \varepsilon_k \E_k (F_k)$. Then if $\mathcal{C}_{2j} = \mathcal{A}_j \otimes \sigma(\varepsilon_0, \cdots, \varepsilon_j)$ and $\mathcal{C}_{2j-1} = \mathcal{A}_j \otimes \sigma(\varepsilon_0, \cdots, \varepsilon_{j-1})$, we have $$f' = \sum_j (\E^{\mathcal{C}_{2j}} - \E^{\mathcal{C}_{2j-1}})(f).$$ Indeed, $\E^{\mathcal{C}_{2j}} (f) = \sum_0^j \varepsilon_k \E_j(F_k)$ and $\E^{\mathcal{C}_{2j-1}} (f) = \sum_0^{j-1} \varepsilon_k \E_j (F_k)$. Hence $(\E^{\mathcal{C}_{2j}} -\E^{\mathcal{C}_{2j-1}})( f ) = \varepsilon_j \E_j (F_j)$. It follows (see the next remark) that $$\| f' \|_{L_s(\mu_\infty \times \Prob; X)} \leq C_s(X) \| f \|_{L_s(\mu_\infty \times \Prob; X)},$$ whence \eqref{stein_umd}.
\end{proof}

\begin{rem}\label{extreme_point}
By an extreme point argument, we have $$\sup_{-1\le \alpha_k \le 1} \| \sum_{k=0}^n \alpha_k df_k \|_{L_s(X)} = \sup_{\varepsilon_k \in \{-1, 1\}} \| \sum_{k=0}^n \varepsilon_k df_k \|_{L_s(X)} .$$ Hence we have $$\sup_{-1\le \alpha_k \le 1} \| \sum_{k=0}^n \alpha_k df_k \|_{L_s(X)} \le C_s(X) \| \sum_{k=0}^n df_k \|_{L_s(X)}
.$$
\end{rem}

\begin{prop}\label{unconditional}
Let $X$ be a $\text{UMD}$ space. Assume that $\{x_i\}_{i \in I}$ is a 1-unconditional basic sequence in $X$. Then for any $1 < s <\infty$, any finite sequence of functions $(\theta_k)_{k \geq 0}$ in $L_s(\Omega, \Prob)$ and any filtration $\mathcal{A}_0 \subset \mathcal{A}_1 \subset \cdots \subset \mathcal{A}_n \subset \cdots$ on $(\Omega, \Prob)$, we have \begin{eqnarray}\label{crucial} \Big \|\sum_k \E_k(\theta_k)  x_{i_k} \Big \|_{L_s(\Omega, \Prob; X)} \leq C_s(X) \Big \| \sum_k \theta_k  x_{i_k}\Big \|_{L_s(\Omega, \Prob; X)}.\end{eqnarray}
\end{prop}
\begin{proof}
For any $i_k$'s, consider the sequence $ (F_k)_{k \geq 0}$ in $L_s(\Omega, \Prob; X)$ defined by $F_k (w)= \theta_k(w) x_{i_k}$. Then $\E_k(F_k) = \E_k(\theta_k)  x_{i_k}$. By the 1-unconditionality of $\{ x_i \}_{i \in I}$, for any fixed choice of signs $\varepsilon_k \in \{ -1, 1\}$ and $w \in \Omega$, we have $$\Big \|\sum_k \varepsilon_k F_k(w) \Big \|_X = \Big \| \sum_k \varepsilon_k \theta_k(w)  x_{i_k} \Big \|_X = \Big \|\sum_k \theta_k(w) x_{i_k} \Big \|_X.$$ It follows that $$ \Big \| \sum_k \varepsilon_k F_k \Big \|_{L_s(\mu_\infty \times  \Prob; X)} = \Big \| \sum_k \theta_k  x_{i_k}\Big \|_{L_s(\Omega, \Prob; X)}.$$ Similarly, we have $$  \Big \| \sum_k \varepsilon_k \E_k(F_k) \Big \|_{L_s(\mu_\infty \times  \Prob; X)} = \Big \| \sum_k \E_k(\theta_k) x_{i_k}\Big \|_{L_s(\Omega, \Prob; X)} .$$ By these equalities, \eqref{crucial} follows from \eqref{stein_umd}. 
\end{proof}

Let $X$ be as in Proposition \ref{unconditional}, $\{ x_i \}_{i \in I}$ is a 1-unconditional basic sequence in $X$. Assume that $\theta_k \in L_\infty(\Omega, \Prob)$. By an application of the contractive inclusions $L_\infty(\Omega, \Prob; X) \subset L_s(\Omega, \Prob; X) \subset L_1(\Omega, \Prob; X)$, we have \begin{eqnarray}\label{S(X)} \Big \|\sum_k \E_k(\theta_k) x_{i_k} \Big \|_{L_1(\Omega, \Prob; X)} \leq C_s(X) \Big \| \sum_k \theta_k  x_{i_k}\Big \|_{L_\infty(\Omega, \Prob; X)}. \end{eqnarray} Hence \begin{eqnarray}\label{S(X)_C_s} S(X; \{ x_i\}) \leq C_s(X)\end{eqnarray} for all $1 < s < \infty$.

\begin{thm}\label{E(F)}
Let $E$ be a Banach space with a 1-unconditional basis $\{ e_i : i \in I \}$, let $F$ be another Banach space. By definition, $E(F)$ is the completion of the algebraic tensor product $E \otimes F$ under the norm defined as follows: if $x = \sum_i e_i \otimes x_i \in E \otimes F$, where $(x_i)$ is a finite supported sequence in $F$, then $$\| x \|_{E(F)} : = \Big \| \sum_i e_i \big\| x_i \big\|_F \Big\|_E .$$ For any fixed family of vectors $\{f_j: j\in J\}$ in $F$, consider the family of vectors $\{e_i \otimes f_j: i\in I, j \in J\}$. Then we have $$S(E(F)) \geq S(E)S(F),$$ where $S(E(F))$, $S(E)$ and $S(F)$ are defined with respect to the mentioned families of vectors respectively.
\end{thm}
\begin{proof}
From the definition, for any $\varepsilon > 0$, there exist finite number of distinct indices $\{ i_k: 1\le k \le N_1\} \subset I$ and $\{ j_n: 1 \le n \le N_2\} \subset J$, and there exist functions $\theta_k \in L_\infty(\Omega', \Prob'), 1 \le k \le N_1$ and functions $\xi_n \in L_\infty(\Omega_0, \Prob_0), 1 \le n \le N_2$ satisfying $$\|\sum_k \theta_k  e_{i_k}\|_{L_\infty(\Omega',\Prob'; E)} \le 1$$ and $$\|\sum_n \xi_n f_{j_n}\|_{L_\infty(\Omega_0, \Prob_0; F)}\le 1$$ such that $$ \Big \| \sum_k\E^{\mathcal{A}_k}(\theta_k) e_{i_k} \Big \|_{L_1(\Omega', \Prob'; E)} \geq S(E) -\varepsilon$$ and $$\Big \| \sum_n \E^{\mathcal{B}_n}(\xi_n)  f_{j_n} \Big \|_{L_1(\Omega_0, \Prob_0; F)} \geq S(F) -\varepsilon.$$ Let $(\Omega, \Prob) = (\Omega' \times \Omega_0^{\N}, \Prob' \otimes \Prob_0^{\otimes \N})$, the general element in $\Omega$ will be denoted by $ w = (w', (w_l)_{l\geq 0})$. Consider the $\sigma$-algebras $\mathcal{F}_{k, n}$ defined on $(\Omega, \Prob)$ by $$\mathcal{F}_{k,n} : = \mathcal{A}_k \otimes \underbrace{\mathcal{B}_\infty \otimes \cdots \otimes \mathcal{B}_\infty}_{k-1 \text{ times }} \otimes \mathcal{B}_n \otimes \mathcal{C}_{\geq k +1},$$ where $\mathcal{B}_\infty = \sigma(\mathcal{B}_n: n \geq 0)$ is a $\sigma$-algebra on $(\Omega_0, \Prob_0)$, $\mathcal{B}_0$ is assumed to be trivial and $\mathcal{C}_{\geq k+1}$ is the trivial $\sigma$-algebra on $(\Omega_0^{\N_{\geq {k+1}}}, \Prob_0^{\N_{\geq k+1}})$. It is easy to check that $\mathcal{F}_{k,n}$ is a filtration with respect to the lexigraphic order, i.e. if $(k, n) < (k',n')$ (that is $k < k'$ or $k = k'$ but $n < n'$), then $\mathcal{F}_{k,n} \subset \mathcal{F}_{k', n'}$. 

Now let us define $h : \Omega \rightarrow E(F)$ by $$h(w) = h(w', (w_l))  = \sum_{k, n} \theta_k(w') \xi_n(w_k) e_{i_k} \otimes f_{j_n}.$$ Let $h_{k, n} (w) = \theta_k(w') \xi_n(w_k)$, then $h = \sum_{k,n} h_{k,n} e_{i_k} \otimes f_{j_n}$. Clearly, we have \begin{eqnarray}\label{cond_exp_h}\E^{\mathcal{F}_{k,n}} (h_{k,n})(w)  = \Big[\E^{\mathcal{A}_k} (\theta_k)\Big] (w') \Big[\E^{\mathcal{B}_n} (\xi_n)\Big](w_k) \quad a.e.. \end{eqnarray} By the 1-unconditionality of $\{ e_i: i \in I \}$, for a.e. $w \in \Omega$, we have \begin{eqnarray*}\| h(w) \|_{E(F)} & = & \Big \| \sum_{k,n} \theta_k(w') \xi_n(w_k) e_{i_k} \otimes f_{j_n} \Big\|_{E(F)}\\ & = & \Big \| \sum_k e_{i_k} \big \| \sum_n \theta_k(w') \xi_n (w_k) f_{j_n} \big \|_F \Big \|_E \\ & = &\Big \| \sum_k e_{i_k} | \theta_k(w')| \big \| \sum_n  \xi_n (w_k) f_{j_n} \big \|_F \Big \|_E  \\ & \leq & \Big \| \sum_k e_{i_k} | \theta_k(w')| \Big \|_E  = \Big \| \sum_k e_{i_k}  \theta_k(w') \Big \|_E  \leq 1.\end{eqnarray*} Hence $\| h \|_{L_\infty(\Omega, \Prob; E(F))} \leq 1$. If we denote $$\widetilde{h} = \sum_{k,n} \E^{\mathcal{F}_{k,n}} (h_{k,n}) e_{i_k} \otimes f_{j_n},$$ then by \eqref{cond_exp_h}, \begin{eqnarray*} \| \widetilde{h}(w) \|_{E(F)} & = &  \Big \| \sum_k e_{i_k} |\E^{\mathcal{A}_k}(\theta_k)(w')| \big \| \sum_n \E^{\mathcal{B}_n}(\xi_n)(w_k) f_{j_n} \big\|_F \Big \|_E. \end{eqnarray*} By Jensen's inequality, we have \begin{eqnarray*}& & \int  \Big \| \sum_k e_{i_k} |\E^{\mathcal{A}_k}(\theta_k)(w')| \big \| \sum_n \E^{\mathcal{B}_n}(\xi_n)(w_k) f_{j_n} \big\|_F \Big \|_E d \Prob_0^{\otimes \N} ((w_l)) \\ & \geq & \Big \| \int \sum_k e_{i_k} |\E^{\mathcal{A}_k}(\theta_k)(w')| \big \| \sum_n \E^{\mathcal{B}_n}(\xi_n)(w_k) f_{j_n} \big\|_F  d \Prob_0^{\otimes \N} ((w_l)) \Big \|_E \\ & = & \Big\| \sum_k e_{i_k} | \E^{\mathcal{A}_k}(\theta_k)(w')| \Big \|_E \cdot \Big\| \sum_n \E^{\mathcal{B}_n}(\xi_n) f_{j_n} \Big \|_{L_1(\Omega_0, \Prob_0; F)} \\ & = &\Big\| \sum_k e_{i_k}  \E^{\mathcal{A}_k}(\theta_k)(w') \Big \|_E \cdot \Big\| \sum_n \E^{\mathcal{B}_n}(\xi_n) f_{j_n} \Big \|_{L_1(\Omega_0, \Prob_0; F)} .\end{eqnarray*} Note that in the last equality,  we used the 1-unconditionality assumption on $\{ e_i: i \in I\}$. By integrating both sides with respect to $\int d \Prob'(w')$, we get \begin{eqnarray*} & & \Big \| \sum_{k, n} \E^{\mathcal{F}_{k,n}} (h_{k,n}) e_{i_k} \otimes f_{j_n} \Big \|_{L_1(\Omega, \Prob; E(F))} \\ & \geq &  \Big\| \sum_k \E^{\mathcal{A}_k}(\theta_k) e_{i_k} \Big \|_{L_1(\Omega', \Prob'; E) }\cdot \Big\| \sum_n \E^{\mathcal{B}_n}(\xi_n) f_{j_n} \Big \|_{L_1(\Omega_0, \Prob_0; F)} \\ & \geq & (S(E)-\varepsilon)(S(F) -\varepsilon).\end{eqnarray*} Therefore $S(E(F)) \geq (S(E)-\varepsilon)(S(F) -\varepsilon)$. Since $\varepsilon > 0$ is arbitrary, it follows that $S(E(F)) \geq S(E)S(F)$ as desired.
\end{proof}

\begin{rem}
If $E$ is a Banach lattice which is $p$-convex and $q$-concave $($see \cite{Classical_Banach_2} for the details$)$ with $1 \leq p \leq q \leq \infty$ and $F$ is a Banach space. Then the preceding proof is valid with $S_{q, p}(E)$ and $S_{q, p}(F)$ defined using \eqref{definition_S(X)} with $L_p$-norm on the left hand side and $L_q$-norm on the right hand side.
\end{rem}

\begin{rem}
Let $1\leq p < q \leq \infty$. If we define $C_{q, p}(X)$ as the best constant $C$ in \eqref{definition_umd} with $L_p$-norm on the left hand side and $L_q$-norm on the right hand side, it is well-known that $X$ is in the $\text{UMD}$ class if and only if $C_{q, p}(X) < \infty$. The preceding argument shows that under the same assumption of Theorem \ref{E(F)}, we have $C_{\infty, 1}(E(F)) \geq S(E) C_{\infty, 1}(F)$. Moreover, if $E$ is $p$-convex and $q$-concave we have $C_{q, p}(E(F)) \geq S_{q, p}(E) C_{q, p}(F)$.
\end{rem}

\begin{lem}\label{baselem}
Suppose that $1 \leq p \ne q \leq \infty$. If $E_1 = \ell_p^{(2)}(\ell_q^{(2)})$, then $$S(E_1) \geq c(p,q) > 1.$$
\end{lem}
\begin{proof}
Denote by $\{e_1^p, e_2^p\}$, $\{e_1^q, e_2^q \}$ the canonical basis of $\ell_p^{(2)}$ and $\ell_q^{(2)}$ respectively .Then $\{ e_1^p\otimes e_1^q, e_1^p \otimes e_2^q, e_2^p \otimes e_1^q, e_2^p\otimes e_2^q \}$ is the canonical 1-unconditional basis of $\ell_p^{(2)}(\ell_q^{(2)})$. Consider the probability space $(D, \mu)$ equipped with the filtration $\{ \phi, D\} \subset \sigma(\varepsilon)$, where $\varepsilon$ is the identity function on $D$. Define a linear map  $T: L_\infty (D; E_1) \rightarrow L_1(D; E_1)$ by setting $$ T\Big[a_{ij}(\varepsilon) e_i^p\otimes e_j^q\Big] = \left\{ \begin{array}{lc} \E(a_{ij}) e_i^p\otimes e_j^q, \text{ if } j = 1 \\ a_{ij}(\varepsilon) e_i^p\otimes e_j^q, \text{ if } j =2 \end{array}.\right. $$ By definition of $S(E_1)$ we have $S(E_1)\geq \| T \|_{L_\infty(D; E_1) \rightarrow L_1(D; E_1)}$. Now for any $a, b$ two scalar functions on $D$ , consider $$f(\varepsilon) = e_1^p \otimes \Big[ a(\varepsilon) e_1^q +  b(\varepsilon) e_2^q \Big]+ e_2^p \otimes \Big[ a(-\varepsilon) e_1^q +  b(-\varepsilon) e_2^q\Big] .$$ Then $$ (Tf)(\varepsilon) = e_1^p \otimes \Big[ \E(a) e_1^q +  b(\varepsilon) e_2^q \Big]+ e_2^p \otimes \Big[ \E(a) e_1^q +  b(-\varepsilon) e_2^q\Big] .$$ If $ p, q$ are both finite, then for any fixed $\varepsilon \in D$, we have \begin{eqnarray*} \| f(\varepsilon) \|_{E_1}\! \! \!  & = &\! \!  \Big \{ (| a(\varepsilon) |^q + | b(\varepsilon) |^q )^{p/q} + (|a(-\varepsilon) |^q + | b(-\varepsilon) |^q)^{p/q}   \Big\}^{1/p} \\ & = &\! \!  \Big \{ (| a(1) |^q + | b(1) |^q )^{p/q} + (|a(-1) |^q + | b(-1) |^q)^{p/q}   \Big\}^{1/p}  \\ & = & \! \! 2^{1/p} \Big \{ \frac{1}{2} (| a(1) |^q + | b(1) |^q )^{p/q} + \frac{1}{2} (|a(-1) |^q + | b(-1) |^q)^{p/q}   \Big\}^{1/p} \\ & = & \! \! 2^{1/p} \Big\{\int (| a(\varepsilon) |^q + | b(\varepsilon) |^q )^{p/q} d \mu (\varepsilon)\Big\}^{1/p} \\ & = & \! \! 2^{1/p} \big\| (a, b) \big\|_{L_p(\mu; \ell_q^{(2)})}.\end{eqnarray*} Similarly, $$\| (Tf )(\varepsilon)\|_{E_1} = 2^{1/p} \big\| (\E a, b) \big\|_{L_p(\mu; \ell_q^{(2)})}.$$ It follows that $$\| f \|_{L_\infty(D; E_1)} = 2^{1/p} \big\| (a, b) \big\|_{L_p(\mu; \ell_q^{(2)})}$$ and $$\| Tf \|_{L_1(D; E_1)} = 2^{1/p} \big\| (\E a, b) \big\|_{L_p(\mu; \ell_q^{(2)})}.$$ Hence \begin{eqnarray}\label{T_norm} \| T \|_{L_\infty(D; E_1) \rightarrow L_1(D; E_1)}  \geq \frac{\| Tf \|_{L_1(D; E_1)}}{\| f \|_{L_\infty(D; E_1)}} = \frac{ \big\| (\E a, b) \big\|_{L_p(\mu; \ell_q^{(2)})}}{\big\| ( a, b) \big\|_{L_p(\mu; \ell_q^{(2)})}}.\end{eqnarray} Similarly, if $q = \infty$ and $p$ is finite, then $$\| f \|_{L_\infty(D; E_1)} = 2^{1/p} \|(a, b)\|_{L_p(\mu; \ell_\infty^{(2)})}$$ and $$\| Tf \|_{L_1(D; E_1)} = 2^{1/p} \big\| (\E a, b) \big\|_{L_p(\mu; \ell_\infty^{(2)})}.$$ If $p = \infty$ and $q$  is finite, then  $\| f \|_{L_\infty(D; E_1)} = \|(a, b)\|_{L_\infty(\mu; \ell_q^{(2)})}$ and $\| Tf \|_{L_1(D; E_1)} = \big\| (\E a, b) \big\|_{L_\infty(\mu; \ell_q^{(2)})}$. Therefore, \eqref{T_norm} holds in full generality. By Proposition \ref{mainlem}, we have $$\| T \|_{L_\infty(D; E_1) \rightarrow L_1(D; E_1)} \geq \| P \| = c(p, q).$$ Hence $S(E_1) \geq c(p, q) > 1$, as announced.
\end{proof}

\begin{rem}
Let $(e_k)_{k\geq 0}$ be the canonical basis of $\ell_p = \ell_p(\N)$, then $S(\ell_p) = 1$. Indeed, if $(\theta_k)_{k \geq 0}$ is a finite sequence of functions, then \begin{eqnarray*} \Big \| \sum_k \E_k(\theta_k) e_k \Big \|_{L_1(\ell_p)} & \leq &\Big \| \sum_k \E_k(\theta_k) e_k \Big \|_{L_p (\ell_p)} = \Big \| (\sum_k | \E_k (\theta_k)|^p)^{1/p} \Big \|_{L_p} \\ & = &   \Big \| \sum_k | \E_k (\theta_k)|^p \Big \|_{L_1}^{1/p} = (\sum_k \big \| \E_k (\theta_k) \big \|_p^p )^{1/p} \\ & \leq &  (\sum_k \big \| \theta_k \big \|_p^p )^{1/p} = \Big \| \sum_k \theta_k e_k \Big \|_{L_p (\ell_p)} \\ & \leq & \Big \| \sum_k \theta_k e_k \Big \|_{L_\infty (\ell_p)} .\end{eqnarray*} 
\end{rem}

\begin{thm}\label{mainthm}
Suppose that $1\leq p, q \leq \infty$. Let $E_1 = \ell_p^{(2)}(\ell_q^{(2)})$ and define by recursion: $E_{n+1} = \ell_p^{(2)}(\ell_q^{(2)}(E_n))$. Then for any $1 < s < \infty$, we have $$C_s(E_n) \geq S(E_n) \geq c(p,q)^n, $$ where $S(E_n)$ is computed with respect to the canonical basis of $E_n$. In particular, if $p \ne q$, then $C_s(E_n)$ has at least an exponential growth with respect to $n$. 
\end{thm}
\begin{proof}
By Theorem \ref{E(F)}, $$S(E_{n+1}) \geq S(\ell_p^{(2)}(\ell_q^{(2)})) S(E_n).$$ By Lemma \ref{baselem}, we have $S(E_{n+1}) \geq c(p, q)S(E_n)$. It follows that $S(E_n) \geq c(p,q)^n$. Since the canonical basis of $E_n$ is 1-unconditional, by \eqref{S(X)_C_s}, for any $1 < s < \infty$, we have $C_s(E_n) \geq S(E_n)$.
\end{proof}

The following simple observation shows that the exponential growth of $C_s(E_n)$ is optimal.
\begin{prop}\label{majoration}
Suppose $1 < p \ne q < \infty$. Let $X$ be a Banach space. Define by recursion: $Y_0 =X$ and $Y_{n+1} = L_p(\T; L_q(\T; Y_n))$. Then for all $1 < s < \infty$, there exists $\chi = \chi(p, q, s)$, such that $$C_s(Y_n) \leq \chi^n C_s(X).$$
\end{prop}
\begin{proof}
We will use the following well-known fact (see e.g. \cite{Burkholder_geo_UMD, Burkholder1}) about $\text{UMD}$ constants: for any $1< r,s< \infty$, there exist $\alpha(r, s)$ and $\beta(r, s)$ such that for all Banach space $X$, \begin{eqnarray}\label{umd_indep_p}\alpha(r, s) C_s(X) \leq C_r(X) \leq \beta(r,s) C_s(X).\end{eqnarray} We will also use the elementary identity $C_s(L_s(X)) = C_s(X)$. Combining these, we have \begin{eqnarray*} C_s(Y_{n+1})& = &C_s(L_p( L_q( Y_n))) \leq \beta(s, p) C_p(L_p(L_q(Y_n)))\\ & =&  \beta(s,p) C_p(L_q(Y_n)) \leq \beta(s,p) \beta(p, q) C_q(L_q(Y_n)) \\ &= &\beta(s,p)\beta(p,q) C_q(Y_n) \leq \beta(s,p)\beta(p,q) \beta(q,s)C_s(Y_n).\end{eqnarray*} Let $\chi = \beta(s,p)\beta(p,q) \beta(q,s)$, then $C_s(E_n) \leq \chi^n C_s(X).$
\end{proof}

\begin{rem}
Even if one of $p,q$ is infinite or equals to 1, then since $\dim(E_n) = 4^n$, we have $C_s(E_n) \lesssim \sqrt{\dim E_n} = 2^n$. Indeed, the Banach-Mazur distance between $E_n$ and $\ell_2^{\dim E_n}$ is $\le \sqrt{\dim E_n}$ $($cf. e.g.\,\cite{Banach_Mazur}$)$.
\end{rem}

%%%%%%%%%%%%%%%%%%%%%%%%%%%%%%%%%%%%%%%%AUMD case

\section{Analytic $\text{UMD}$ constants}\label{analytic_umd}
The main idea in \S\ref{mainresults} can be easily adapted for treating the analytic $\text{UMD}$ property. In this section, all spaces are over $\C$.

Denote the general element in $\T^\N$ be $z = (z_n)_{n\geq 0}$ and let $m_\infty = m^{\otimes \N}$ be the Haar measure on $\T^\N$. Recall the canonical filtration on $(\T^{\N}, m_\infty)$ defined by $$\sigma(z_0) \subset \sigma(z_0, z_1) \subset \cdots \subset \sigma(z_0, z_1, \cdots, z_n)\subset \cdots.$$ From now on, we will denote $\mathcal{G}_n = \sigma(z_0, z_1, \cdots, z_n)$. Recall that $H_s(\T^\N)$ is the subspace of $L_s(\T^\N, m_\infty)$ consisting of limit values of Hardy martingales, i.e. $f \in H_s(\T^\N)$ if and only if $f \in L_s(\T^\N, m_\infty)$ and the associated martingale $(\E^{\mathcal{G}_n} f)_{n \geq 0}$ is a Hardy martingale. For convenience, we always assume $z_0 \equiv 1$ such that $\mathcal{G}_0$ is a trivial $\sigma$-algebra.

\begin{defn}
Let $X$ be a Banach space and let $\{ x_i\}_{i \in I}$ be a family of vectors in $X$. The number $S^a(X; \{ x_i \})$ is defined to be the best constant $C$ such that for any $N\in \N$ and any finite sequence of functions $(\theta_k)_{k=0}^{N}$ in $H_\infty(\T^\N)$, we have $$ \Big \| \sum_k \E^{\mathcal{G}_k} (\theta_k) x_{i_k} \Big \|_{L_1(m_\infty; X)} \leq C \Big \| \sum_k \theta_k x_{i_k} \Big \|_{L_\infty(m_\infty; X)}$$ If there does not exist such constant,  we set $S^a(X; \{x_i\}) = \infty$. 
\end{defn}
If $\{ x_i\}$ is clear from the context, then $S^a(X; \{ x_i \})$ will be simplified as $S^{a}(X)$.

The Stein type inequality still holds in this setting, more precisely, we have
\begin{prop}\label{stein_aumd}
Let $X$ be an $\text{AUMD}$ space. For any $1\leq s < \infty$, let $(F_k)_{k \geq 0}$ be an arbitrary finite sequence in $H_s(\T^\N; X)$. Then we have \begin{eqnarray}\label{aumd_stein_ineq} \Big \| \sum_k \zeta_k \E^{\mathcal{G}_k} (F_k) (z) \Big\|_{L_s( X)} \leq C_s^a(X) \Big \| \sum_k \zeta_k F_k(z) \Big \|_{L_s(X)},\end{eqnarray} where $\zeta = (\zeta_k)_{k \geq 0}$ is an independent copy of $z = (z_k)_{k\geq 0}$ and $L_s(X) = L_s(\T_z^\N \times \T_\zeta^\N, m_\infty \times m_\infty; X)$.
\end{prop}
\begin{proof}
Consider the filtration on $\T_z^\N \times \T_\zeta^\N$ defined by $\mathcal{B}_{2j} = \sigma(z_0, \cdots, z_j)\otimes \sigma(\zeta_0, \cdots, \zeta_j)$ and $\mathcal{B}_{2j-1} = \sigma(z_0, \cdots, z_j) \otimes \sigma(\zeta_0, \cdots, \zeta_{j-1})$. Then $f = \sum_k \zeta_k F_k(z)$ is a Hardy martingale with respect to the above filtration. Let $f' = \sum_k \zeta_k \E^{\mathcal{G}_k}(F_k)$. Then we have $f' = \sum_j (\E^{\mathcal{B}_{2j}} - \E^{\mathcal{B}_{2j-1}})(f)$. It follows (see Remark \ref{extreme_point}) that $\| f' \|_{L_s(X)} \leq C_s^{a}(X) \| f \|_{L_s(X)}$, whence \eqref{aumd_stein_ineq}.
\end{proof}

\begin{prop}\label{aumd_unconditional}
Let $X$ be an $\text{AUMD}$ space. Assume that $\{x_i\}_{i \in I}$ is a 1-unconditional basic sequence in $X$. Then for any $1 \leq s < \infty$ and any finite sequence of functions $(\theta_k)_{k \geq 0}$ in $H_s(\T^\N)$, \begin{eqnarray*} \Big \|\sum_k \E^{\mathcal{G}_k}(\theta_k)  x_{i_k} \Big \|_{L_s(m_\infty; X)} \leq C_s^{a}(X) \Big \| \sum_k \theta_k  x_{i_k}\Big \|_{L_s(m_\infty; X)}.\end{eqnarray*}
\end{prop}
\begin{proof}
It follows verbatim the proof of Proposition \ref{unconditional}.
\end{proof}

Let $X$ be as in Proposition \ref{aumd_unconditional}, $\{ x_i \}$ is a 1-unconditional basic sequence in $X$. Then for all $1 \leq s < \infty$, we have \begin{eqnarray*} \Big \|\sum_k \E^{\mathcal{G}_k}(\theta_k)  x_{i_k} \Big \|_{L_1(m_\infty; X)} \leq C_s^{a}(X) \Big \| \sum_k \theta_k  x_{i_k}\Big \|_{L_\infty(m_\infty; X)}.\end{eqnarray*} Hence $$ S^{a}(X; \{x_i\}) \leq C_s^{a}(X)$$ for all $1 \leq s < \infty$.

\begin{thm}\label{mainthm_analytic}
Let $E$ be a Banach space with a 1-unconditional basis $\{ e_i : i \in I \}$, let $F$ be another Banach space. Let $E(F)$ be defined as in Theorem \ref{E(F)}. For any fixed family of vectors $\{ f_j : j \in J\}$ in $F$, consider the family of vectors $\{ e_i \otimes f_j: i \in I, j \in J\}$ in $E(F)$, then we have $$S^{a}(E(F)) \geq S^{a}(E)S^{a}(F),$$ where $S^{a}(E(F))$, $S^{a}(E)$ and $S^{a}(F)$ are defined with respect to the mentioned families of vectors respectively.
\end{thm}
\begin{proof}
The proof is similar to the proof of Theorem \ref{E(F)}. We mention the slight difference concerning the filtration. Consider the infinite tensor product $L_\infty(\T^{\N}) \otimes L_\infty(\T^{\N}) \otimes \cdots$, define $$z_{k, n} = \underbrace{1 \otimes \cdots \otimes 1}_{k \quad \text{times}} \otimes z_n \otimes 1 \otimes \cdots, \text{if } n \ge 1$$ and $$z_{k, 0} = z_k \otimes 1 \otimes 1 \otimes \cdots.$$
Then the filtration defined by $\mathcal{F}^{a}_{k,n} : = \sigma\left( z_j: j \le (k, n)\right)$ is an analytic filtration, where the order on $\N \times \N$ is the lexigraphic order as defined in the proof of Theorem \ref{E(F)}. This filtration plays the role similar to that of $(\mathcal{F}_{k, n})_{k,n}$ in the proof of Theorem \ref{E(F)}. Note that we may restrict to the functions $\theta_k, \xi_n$ depending only on finitely many variables. Thus only a finite subset of $\N \times \N$ is used.
\end{proof}

The following lemma requires slightly more efforts than Lemma \ref{baselem}.
\begin{lem}\label{baselem_aumd}
Suppose that $1 \leq p \neq q < \infty$. If $E_1 = \ell_p^{(2)}(\ell_q^{(2)})$, then $$S^{a}(E_1) \geq \kappa(p,q) > 1.$$
\end{lem}
\begin{proof}
We will use the notations in the proof of Lemma \ref{baselem}. Define a linear map $U: H_\infty(\T, m; E_1) \rightarrow H_1(\T, m; E_1)$ by $$ U \Big[a_{ij}(z) e_i^p\otimes e_j^q\Big] = \left\{ \begin{array}{lc} \E(a_{ij}) e_i^p\otimes e_j^q, \text{ if } j = 1 \\ a_{ij}(z) e_i^p\otimes e_j^q, \text{ if } j =2 \end{array}.\right.$$ If $C = \| U \|_{H_\infty(E_1) \rightarrow H_1(E_1)}$, then $S^{a}(E_1) \geq C$. By definition, for any $a, b, c, d$ functions in $H_\infty(\T)$, we have \begin{eqnarray}\label{abcd} & &  \int_\T \Big\{ (|\E a|^q + | b(z) |^q)^{p/q} + (|\E c |^q + | d(z) |^q)^{p/q} \Big \}^{1/p} dm(z) \\ \nonumber & \leq & C \sup_{ z \in \T}  \Big\{ (| a(z)|^q + | b(z) |^q)^{p/q} + (| c(z) |^q + | d(z) |^q)^{p/q} \Big \}^{1/p}. \end{eqnarray} Note that if $a, c$ are outer functions, then by \eqref{geo_mean_outer}, we have $| \E a | =  M(|a|)$ and $| \E c | = M(|c|)$. So for any functions $a, b, c, d \in H_\infty(\T)$ such that $a, c$ are outer, we have 
\begin{eqnarray}\label{outer} & &  \int_\T \Big\{ (M (|a|)^q + | b(z) |^q)^{p/q} + (M(| c |)^q + | d(z) |^q)^{p/q} \Big \}^{1/p} dm(z) \\ \nonumber & \leq & C \sup_{ z \in \T}  \Big\{ (| a(z)|^q + | b(z) |^q)^{p/q} + (| c(z) |^q + | d(z) |^q)^{p/q} \Big \}^{1/p}. \end{eqnarray} 
By the classical Szeg\"{o}'s condition, if $a', b', c', d'$ are functions in $L_\infty(\T)$ which are bounded from below, then there are outer functions $a, b, c, d \in H_\infty(\T)$, such that $ |a'| = |a|, |b'| = |b|, |c'| = |c|, |d'| = |d|$ a.e.. Hence \eqref{outer} still holds for any 2-valued non-vanishing functions $a, b, c, d \in L_\infty(\T)$ (note that for a function taking only two values, non-vanishing is the same as bounded from below). By approximation, we can further relax the non-vanishing condition on $b, d$. Now consider any measurable partition $\T = A \dot{\cup} B$, such that $m(A) = m(B) = \frac{1}{2}$. If $a = u \chi_A + v \chi_B$, $c = v \chi_A + u \chi_B$, $b = w \chi_A + t \chi_B$ and $d = t \chi_A + w \chi_B$, then it is easy to check that for all $z \in \T$, we have  \begin{eqnarray*} & & \Big\{ (| a(z)|^q + | b(z) |^q)^{p/q} + (| c(z) |^q + | d(z) |^q)^{p/q} \Big \}^{1/p}  \\ &=&  \Big \{ (|u|^q + | w |^q)^{p/q} + (|v|^q + | t |^q)^{p/q}\Big\}^{1/p} \\
&=& 2^{1/p}\Big\{\int_\T (| a |^q + |b|^q)^{p/q} dm\Big\}^{1/p}.\end{eqnarray*} Similarly for all $z \in \T$, we have \begin{eqnarray*}& & \Big\{ (M(| a|)^q + | b(z) |^q)^{p/q} + (M(| c |)^q + | d(z) |^q)^{p/q} \Big \}^{1/p} \\ &=&  2^{1/p}\Big\{\int_\T (M(| a |)^q + |b|^q)^{p/q} dm\Big\}^{1/p}.\end{eqnarray*} Substituting these equalities to \eqref{outer}, we get $$\Big\{\int_\T (M(| a |)^q + |b|^q)^{p/q} dm\Big\}^{1/p} \leq C \Big\{\int_\T (| a |^q + |b|^q)^{p/q} dm\Big\}^{1/p}.$$ By Proposition \ref{geo_mean}, we have $C \geq \kappa(p,q)$. This completes the proof.
\end{proof}

\begin{thm}\label{mainthm_aumd}
Suppose that $1 \leq p\ne q < \infty$. If $E_n$'s are defined as in Theorem \ref{mainthm}, then for any $1 \leq s < \infty$, we have $$C_s^{a}(E_n) \geq S^{a}(E_n) \geq \kappa(p,q)^n. $$ Moreover, if $1< p, q< \infty$, then there exists $\kappa_2 = \kappa_2(p,q,s)$, such that $$C_s^{a} (E_n) \leq \kappa_2^n.$$
\end{thm}
\begin{proof}
The first part of proof is identical to the proof of Theorem \ref{mainthm}. The second part follows from the fact that $C_s^{a}(E_n) \le C_s(E_n)$ and Proposition \ref{majoration}.  
\end{proof}

%%%%%%%%%%%%%%%%%%%%%%%%%%%%%%%%%%%%Construction and discussions

\section{Construction and further discussions}
For the sake of clearness, we introduce the family $X_n(p, q)$, which is defined as follows: Let $X_0(p,q)= \R$, and define by recursion that $$X_{n+1}(p, q) = L_p(D, \mu; L_q(D, \mu; X_n(p,q))).$$ In the complex case, $X^{\C}_n(p,q)$ is defined similarly. 

Obviously, $X_n(p,q)$ is isometric to $E_n$ defined in the previous sections using $p, q$. Our main purpose for introducing $X_n$'s is the existence of canonical isometric inclusion $X_n(p,q) \subset X_{n+1}(p,q)$. By these inclusions, the union $\cup_n X_n(p,q)$ is a normed space and its completion will be denoted by $X(p,q)$. We have $$X(p,q): =  \overline{\cup_n X_n(p,q)} \simeq \lim_{\longrightarrow} X_n(p,q),$$ where the last term is the inductive limit of $X_n(p,q)$'s associated to the canonical inclusions. In the complex case, $X^{\C}(p,q)$ is defined similarly. 
\begin{rem}
If $1\le p = q < \infty$, then $X(p,p)$ is the real space $L^p_{\R}(D^\N, \mu^{\otimes \N})$ and $X^{\C}(p,p)$ is the complex space $L_{\C}^p(D^\N, \mu^{\otimes \N})$.
\end{rem}

We have the following complex interpolation result.
\begin{prop}\label{inter}
Let $1< p_0, p_1, q_0, q_1< \infty$ and $0 < \theta < 1$. Then we have the following isometric isomorphism: $$X^{\C}(p_\theta, q_\theta) = [ X^{\C}(p_0, q_0), X^{\C}(p_1, q_1)]_\theta,$$ with $\frac{1}{p} = \frac{\theta}{p_1} + \frac{1-\theta}{p_0}$ and $\frac{1}{q} = \frac{\theta}{q_1} + \frac{1-\theta}{q_0}$.
\end{prop}
\begin{proof}
Note that $X(p,q)$ is a Banach lattice of functions on $(D^\N, \mu^{\otimes \N})$. Clearly, $X(p, q)$ is $\min(p,q)$-convex and $\max(p,q)$-concave in the sense of \S 1.d in \cite{Classical_Banach_2}, and hence by  Theorem 1.f.1 (p.\,80) and Proposition 1.e.3 (p.\,61) in \cite{Classical_Banach_2} it is reflexive. Then the above result is a particular case of a classical formula going back to Calder\'on (\cite{Calderon_Interpolation}, p.\,125).
\end{proof}

Recall that a Banach space $X$ over the complex field is $\theta$-Hilbertian ($0 \le \theta \le 1$) if there exists an interpolation pair $(X_0, X_1)$ of Banach spaces such that $X$ is isometric with $[X_0, X_1]_\theta$ and $X_1$ is a Hilbert space.
\begin{cor}
Let $1< p \ne q< \infty$. Then $X(p, q)$ is non-$\text{UMD}$ and $X^{\C}(p, q)$ is non-$\text{AUMD}$. Moreover, there exists $0 < \theta< 1$ such that $X^\C(p, q)$ is $\theta$-Hilbertian. In particular, $X^{\C}(p, q)$ and a fortiori $X(p, q)$ is super-reflexive.
\end{cor}
\begin{proof}
It follows easily from Theorem \ref{mainthm} and Theorem \ref{mainthm_aumd} that $X(p, q)$ is non-$\text{UMD}$ and $X^{\C}(p, q)$ is non-$\text{AUMD}$. 

For $0< \theta < 1$ small enough, such that $\max( \frac{1/p- \theta/2}{1-\theta},  \frac{1/q- \theta/2}{1-\theta}) < 1$, we can find $1< \tilde{p}, \tilde{q} < \infty$ satisfying the equalities: $$ \frac{1}{p} = \frac{\theta}{2} + \frac{1-\theta}{\tilde{p}}, \quad \frac{1}{q} = \frac{\theta}{2} + \frac{1-\theta}{\tilde{q}}.$$ By Proposition \ref{inter}, we have $$ X^\C(p, q) = [X^\C(\tilde{p}, \tilde{q}), X^\C(2, 2)]_\theta.$$ Since $X^\C(2, 2) = L^2_{\C}(D^\N, \mu^{\otimes \N})$ is Hilbertian, $X^\C(p, q)$ is $\theta$-Hilbertian.  The super-reflexivity of $X^\C(p,q)$ follows from the well-known fact that any $\theta$-Hilbertian space is super-reflexive for $\theta > 0$ (cf.\cite{Pisier_interpolation_lattice}).
\end{proof}

\begin{rem}
Let $1 < p\neq q < \infty$. For any $ 0 < \eta < 1$, let $\frac{1}{p_\eta} = \frac{1-\eta}{p} + \frac{\eta}{q}$ and $\frac{1}{q_\eta} = \frac{1-\eta}{q} + \frac{\eta}{p}$. By Proposition \ref{inter}, we have $$X^\C(p_\eta, q_\eta) = [X^{\C}(p, q), X^\C(q, p)]_\eta. $$ Note that in this interpolation scale, there is only one $\text{UMD}$ space corresponding to $\eta = \frac{1}{2}$.  
\end{rem}

For futher discussions, let us now turn to the non-atomic case and modify slightly the definitions. For any $1 < p, q < \infty$, consider the family of spaces $Z_n = Z_n(p,q)$ defined by recursion: $Z_0 = \C$ and $Z_{n+1} = Z_n(L_p(\T, m; L_q(\T, m))$. From the definition, we have $$Z_n(p,q) \subset Z_{n+1}(p,q).$$ Thus we can define $$Z(p, q) = \lim_{\longrightarrow}Z_n(p,q).$$ To avoid ambiguity, let us emphasize the inclusions $Z_n(p,q) \subset Z_{n+1}(p,q)$ used to define the inductive limit. For simplicity of notations, we will write $L_{p_1} L_{p_2} = L_{p_1}(L_{p_2})$, $L_{p_1}L_{p_2} L_{p_3} = L_{p_1}(L_{p_2}(L_{p_3}))$, etc. With these notations, one can easily see the difference between $X_n$ and $Z_n$ as follows: $$X_{n + 1} = L_p(L_q(X_n)) = L_p L_q \underbrace{L_p L_q \cdots L_p L_q}_{X_n},$$ where $L_p = L_p(D, \mu)$ and $L_q = L_q(D, \mu)$ are two dimensional. And
$$Z_{n + 1} = Z_n(L_p(L_q)) =   \underbrace{L_p L_q \cdots L_p L_q}_{Z_n} L_p L_q,$$ where $L_p = L_p(\T, m)$ and $L_q = L_q(\T, m)$.

\begin{rem}
The main purpose of introducing the spaces $Z_n(p, q)$ is that we have lattice isometric isomorphisms $L_p(Z_n(p,q)) \simeq Z_n(p,q)$ for all $n$ and moreover, these isomorphisms are compatible with the inclusion of $Z_n(p,q) \subset Z_{n+1}(p,q)$ (the word ``compatible'' will be explained by a commutative diagram in the sequel) and this will be used to show some additional properties for $Z(p,q)$. The family of $X_n(p,q)$'s shares the property of having lattice isometric isomorphisms $L_p(X_n(p,q)) \simeq X_n(p,q)$ for all $n$, but the isomorphisms are not compatible with the inclusions $X_n(p,q) \subset X_{n+1}(p,q)$.
\end{rem}

The $Z(p,q)$'s are Banach lattices of functions on the infinite torus $\T^\N$, they have the following properties.
\begin{prop}
Let $1< p, q < \infty$. We have isomorphisms $$Z(p, q) \simeq Z(q,p)$$ and $$L_p(Z(p,q)) \simeq L_q(Z(p,q)).$$ If $p \ne q$, then $Z(p,q)$ does not have unconditional basis.
\end{prop}
\begin{proof}
Since $L_p(\T)$ and $L_p(\T \times \T)$ are isometric as Banach lattices, we have isometric isomorphisms which are compatible with the inclusions $Z_n \subset Z_{n+1}$, that is we have the commutative diagram $$\begin{CD}Z_n(p,q) @> \text{ inclusion }>> Z_{n+1}(p,q) \\ @V\text{ isometric }V\simeq V @V\simeq V \text{ isometric} V \\L_p(Z_{n}(p,q)) @ >\text{ inclusion } >> L_p(Z_{n+1}(p,q)).\end{CD}$$By taking Banach space inductive limit, we have $$Z(p,q) \xrightarrow[\text{isometric}]{\simeq} L_p(Z(p,q)).$$ If $p \ne q$, then $Z(p,q)$ and hence $L_p(Z(p,q))$ is non-$\text{UMD}$. By a result of D.J. Aldous (see \cite{uncond_bases}, Proposition 4), $Z(p,q)$ has no unconditional basis. 

It is easy to see that $Z(p,q)$ and $Z(q,p)$ complementably embed into each other. Since $\ell_p^{(2)}(L_p) = L_p$ as Banach lattices, we have $$\ell_p^{(2)}(L_p(Z(p,q))) = L_p(Z(p,q)).$$ Moreover, since $L_p(Z(p,q)) = Z(p,q)$, the above isometry implies that as Banach space $Z(p,q) = Z(p,q) \oplus Z(p,q)$. Similarly, $Z(q,p) = Z(q,p) \oplus Z(q, p)$. By the classical Pe\l{}cy\'nski decomposition method, we have $Z(p, q) \simeq Z(q, p)$. Hence $$\qquad \quad L_p(Z(p,q)) = Z(p,q) \simeq Z(q,p) = L_q(Z(q, p)) \simeq L_q(Z(p,q)).$$
\end{proof}

Let $(p_i)_{i\geq 1}$ be a sequence of real numbers such that $1 < p_i < \infty$. Define $$X[(p_i)] = \lim_{\longrightarrow} L_{p_n} \cdots L_{p_2} L_{p_1}$$ and $$Z[(p_i)]  = \lim_{\longrightarrow} L_{p_1} L_{p_2} \cdots L_{p_n}.$$ 
\begin{prob}
Under which condition is $X[(p_i)]$ or $Z[(p_i)]$ in the $\text{UMD}$ class ?
\end{prob}

We have the following observations on the necessary condition:
\begin{itemize}
\item[(i)]
A trivial necessary condition is that there exist $1< p_0, p_\infty <\infty$, such that $p_0 \leq p_i \leq p_\infty$ for all $i \geq 1$.
\item[(ii)] If the above condition is satisfied, then the sequence $(p_i)$ has at least one cluster point $1<p<\infty$. Then a necessary condition is that the sequence has only one cluster point, i.e.  $\lim_{i\to \infty} p_i = p$. Indeed, assume that the sequence $(p_i)$ has two cluster points $1< p \ne q < \infty$, so that there exist two subsequences of $(p_i)$ which tend to $p,q$ respectively. Then one can easily show that by Theorem \ref{mainthm}, both $X[(p_i)]$ and $Z[(p_i)]$ are non-$\text{UMD}$ (they are in fact non-$\text{AUMD}$).
\item[(iii)] Now the speed of convergence of $(p_i)$ will play a role. Since $\ell_{p_1}^{(2)}(\ell_{p_2}^{(2)}(\cdots (\ell_{p_n}^{(2)})\cdots ))$ embeds isometrically into $L_{p_1}L_{p_2} \cdots L_{p_n}$. A necessary condition for $Z[(p_i)]$ to be $\text{UMD}$ is $\prod_i c(p_{2i}, p_{2i+1}) < \infty$. Similarly, it is necessary that $\prod_i c(p_{2i+1}, p_{2i+2}) < \infty$. Combining these, a necessary condition for $Z[(p_i)]$ to be in the $\text{UMD}$ class is $$\prod_i c(p_i, p_{i+1}) < \infty.$$ The same statement remains true for $X[(p_i)]$. Note that by \eqref{constant}, $c(p_i, p_{i+1})>1$ if $p_i \ne p_{i+1}$.
\end{itemize}

Intuitively, if $p_i$ tends to $p$ sufficiently fast, then both $X[(p_i)]$ and $Z[(p_i)]$ are in the $\text{UMD}$ class. The author obtained some partial results in this direction, which will be treated elsewhere.

\begin{rem}\label{monotone_rem}
Let $1 < p < q < \infty$. We have the following Banach lattices isometries $$L_pL_q = L_pL_pL_q, \quad L_pL_q =  L_pL_qL_q.$$ Since $L_pL_rL_q$ is an interpolation space between $ L_pL_pL_q$ and $L_pL_qL_q$ for any $p \leq r \leq q$, the $\text{UMD}_s$ constant of $L_pL_rL_q$ is actually the same as that of $L_p(L_q)$. The same argument shows that $L_pL_uL_rL_vL_q$ has the same $\text{UMD}_s$ constant with $L_pL_q$, provided $p \leq u \leq r \leq v \leq q$. More generally, if $ ( p_i)_{i=1}^n$ is a finite sequence, assume that $(p_i)_{i=k}^{l}$ is consecutive monotone (non-increasing or non-decreasing) subsequence, then $L_{p_1} \cdots L_{p_k} \cdots L_{p_l} \cdots L_{p_n}$ and $L_{p_1}\cdots L_{p_k}L_{p_l}\cdots L_{p_n}$ have the same $\text{UMD}_s$ constant for all $1 < s < \infty$.
\end{rem}

Our results have some applications in the non-commutative setting, i.e. on the operator space $\text{UMD}$ property, which will appear in a future publication.

\section*{Acknowledgements}
This work was carried out while the author was visiting at Texas A\&M University. The author would like to acknowledge the hospitality provided by Department of Mathematics of Texas A\&M. He would like to thank his advisor G. Pisier for suggesting this problem and for the constant and stimulating discussions. The author also appreciates the careful review of the paper by the referee who suggested many changes to enhance the readability of the paper.


\begin{thebibliography}{10}

\bibitem{uncond_bases}
D.~J. Aldous.
\newblock Unconditional bases and martingales in {$L_{p}(F)$}.
\newblock {\em Math. Proc. Cambridge Philos. Soc.}, 85(1):117--123, 1979.

\bibitem{Bourgain}
J.~Bourgain.
\newblock Some remarks on {B}anach spaces in which martingale difference
  sequences are unconditional.
\newblock {\em Ark. Mat.}, 21(2):163--168, 1983.

\bibitem{Bourgain_UMD}
J.~Bourgain.
\newblock On martingales transforms in finite-dimensional lattices with an
  appendix on the {$K$}-convexity constant.
\newblock {\em Math. Nachr.}, 119:41--53, 1984.

\bibitem{Bourgain_H1_BMO}
J.~Bourgain.
\newblock Vector-valued singular integrals and the {$H^1$}-{BMO} duality.
\newblock In {\em Probability theory and harmonic analysis ({C}leveland,
  {O}hio, 1983)}, volume~98 of {\em Monogr. Textbooks Pure Appl. Math.}, pages
  1--19. Dekker, New York, 1986.

\bibitem{Burkholder_geo_UMD}
D.~L. Burkholder.
\newblock A geometrical characterization of {B}anach spaces in which martingale
  difference sequences are unconditional.
\newblock {\em Ann. Probab.}, 9(6):997--1011, 1981.

\bibitem{Burkholder1}
D.~L. Burkholder.
\newblock A geometric condition that implies the existence of certain singular
  integrals of {B}anach-space-valued functions.
\newblock In {\em Conference on harmonic analysis in honor of {A}ntoni
  {Z}ygmund, {V}ol. {I}, {II} ({C}hicago, {I}ll., 1981)}, Wadsworth Math. Ser.,
  pages 270--286. Wadsworth, Belmont, CA, 1983.

\bibitem{Burkholder3}
Donald~L. Burkholder.
\newblock Martingales and singular integrals in {B}anach spaces.
\newblock In {\em Handbook of the geometry of {B}anach spaces, {V}ol. {I}},
  pages 233--269. North-Holland, Amsterdam, 2001.

\bibitem{Calderon_Interpolation}
A.-P. Calder{\'o}n.
\newblock Intermediate spaces and interpolation, the complex method.
\newblock {\em Studia Math.}, 24:113--190, 1964.

\bibitem{Garling_UMD}
D.~J.~H. Garling.
\newblock On martingales with values in a complex {B}anach space.
\newblock {\em Math. Proc. Cambridge Philos. Soc.}, 104(2):399--406, 1988.

\bibitem{Classical_Banach_2}
Joram Lindenstrauss and Lior Tzafriri.
\newblock {\em Classical {B}anach spaces. {II}}, volume~97 of {\em Ergebnisse
  der Mathematik und ihrer Grenzgebiete}.
\newblock Springer-Verlag, Berlin, 1979.
\newblock Function spaces.

\bibitem{Pisier_superR_nonUMD}
G.~Pisier.
\newblock Un exemple concernant la super-r\'eflexivit\'e.
\newblock In {\em S\'eminaire {M}aurey-{S}chwartz 1974--1975: {E}spaces
  {$L^{p}$}\ applications radonifiantes et g\'eom\'etrie des espaces de
  {B}anach, {A}nnexe {N}o. 2}, page~12. Centre Math. \'Ecole Polytech., Paris,
  1975.

\bibitem{Pisier_interpolation_lattice}
G.~Pisier.
\newblock Some applications of the complex interpolation method to {B}anach
  lattices.
\newblock {\em J. Analyse Math.}, 35:264--281, 1979.

\bibitem{Rubio_de_Francia_mart}
Jos{\'e}~L. Rubio~de Francia.
\newblock Martingale and integral transforms of {B}anach space valued
  functions.
\newblock In {\em Probability and {B}anach spaces ({Z}aragoza, 1985)}, volume
  1221 of {\em Lecture Notes in Math.}, pages 195--222. Springer, Berlin, 1986.

\bibitem{Banach_Mazur}
Nicole Tomczak-Jaegermann.
\newblock {\em Banach-{M}azur distances and finite-dimensional operator
  ideals}, volume~38 of {\em Pitman Monographs and Surveys in Pure and Applied
  Mathematics}.
\newblock Longman Scientific \& Technical, Harlow, 1989.

\end{thebibliography}
\end{document}